\documentclass[12pt, reqno]{amsart}
\pdfoutput=1

\usepackage{amssymb,amsmath,amsthm}
\usepackage{xparse} 
\usepackage{hyperref}
\usepackage{pgf,tikz}

\newtheorem{thm}{Theorem}
\newtheorem{definition}{Definition}
\newtheorem{claim}{Claim}
\newtheorem{lemma}[thm]{Lemma}

\theoremstyle{remark}
\newtheorem{construction}{Construction}
\newtheorem*{remark}{Remark}

\newcommand{\abs}[1]{\left\lvert{#1}\right\rvert}

\DeclareMathOperator{\Def}{Def}
\DeclareMathOperator{\ex}{ex}
\DeclareMathOperator{\exlin}{ex_{lin}}

\title{The exact linear Tur\'an number of the sail}

\author{Beka Ergemlidze}
\address{Beka Ergemlidze, Department of Mathematics and Statistics, University of South Florida,
Tampa, Florida 33620, USA.}
\email{\texttt{beka.ergemlidze@gmail.com}}
\author{Ervin Gy\H{o}ri}
\address{Ervin Gy\H{o}ri, Alfr\'ed R\'enyi Institute of Mathematics, Budapest.}
\email{\texttt{gyori.ervin@renyi.mta.hu}} 

\author{Abhishek Methuku}
\address{Abhishek Methuku,  School of Mathematics,
 University of Birmingham,
Edgbaston,
Birmingham B15 2TT,
United Kingdom.}
\email{\texttt{abhishekmethuku@gmail.com}}

\thanks{The research leading to these results was supported by the EPSRC, grant no. EP/S00100X/1 (A. Methuku). The research of all authors was also partially supported by the National
Research, Development and Innovation Office NKFIH, grants K116769, K132696.
}

\date{\today}

\usepackage{fancyhdr}
\begin{document}

\begin{abstract}
A hypergraph is \emph{linear} if any two of its edges intersect in at most one vertex. 
The sail (or $3$-fan) $F^3$ is the $3$-uniform linear hypergraph consisting of $3$ edges $f_1, f_2, f_3$ pairwise intersecting in the same vertex $v$ and an additional edge $g$ intersecting all $f_i$ in a vertex different from $v$. The \emph{linear Tur\'an number} $\exlin(n, F^3)$ is the maximum number of edges in a $3$-uniform linear hypergraph on $n$ vertices that does not contain a copy of $F^3$.

\vspace{1mm}

F\"{u}redi and Gy\'arf\'as proved that if $n = 3k$, then $\exlin(n, F^3) = k^2$ and the only extremal hypergraphs in this case are transversal designs. They also showed that if $n = 3k+2$, then $\exlin(n, F^3) = k^2+k$, and the only extremal hypergraphs are truncated designs (which are obtained from a transversal design on $3k+3$ vertices with $3$ groups by removing one vertex and all the hyperedges containing it) along with three other small hypergraphs. However, the case when $n =3k+1$ was left open.

\vspace{1mm}

In this paper, we solve this remaining case by proving that $\exlin(n, F^3) = k^2+1$ if $n = 3k+1$, answering a question of F\"{u}redi and Gy\'arf\'as. We also characterize all the extremal hypergraphs. The difficulty of this case is due to the fact that these extremal examples are rather non-standard. In particular, they are not derived from transversal designs like in the other cases.
\end{abstract}

\maketitle

\section{Introduction}

An $r$-uniform hypergraph (or an $r$-graph) $H = (V, E)$ consists of a set $V$ of vertices and a set $E$ of edges, where each edge is a subset of $V$. A hypergraph is $r$-partite if its vertices can be partitioned into $r$ parts so that each edge has exactly one vertex from each part. $3$-uniform hypergraphs are also called \emph{triple systems}. A hypergraph is \emph{linear} if any two of its edges intersect in at most one vertex. In design theory, $3$-uniform linear hypergraphs are called \emph{partial triple systems} and small fixed partial triple systems are called \emph{configurations} (see \cite{Colbourn}).

\vspace{2mm}

Suppose $F$ is an $r$-uniform hypergraph. The \emph{Tur\'an number} $\ex(n,F)$ is the maximum number of edges in an $r$-graph on $n$ vertices that does not contain a copy of $F$. The \emph{linear Tur\'an number} $\exlin(n, F)$ is the maximum number of edges in an $r$-uniform \emph{linear} hypergraph on $n$ vertices that does not contain a copy of $F$. The linear $r$-graphs with $\exlin(n, F)$ edges are called \emph{extremal hypergraphs}.

\vspace{2mm}

Linear Tur\'an numbers of linear cycles have been studied. An $r$-uniform linear cycle of length $\ell$, $C_{\ell}^r$ is a hypergraph with edges $e_1, e_2, \ldots, e_{\ell}$ of size $r$ such that $\abs{e_i \cap e_{i+1}} =1$ for each $i \in [\ell - 1]$, $\abs{e_{\ell} \cap e_{1}} =1$, and $e_i \cap e_j = \emptyset$ for all other pairs $i, j$, $i \not = j$. Determining $\exlin(n, C_{3}^3)$ is equivalent to the famous $(6,3)$-problem, which asks for the maximum number of edges in a $3$-graph on $n$ vertices in which no $6$ vertices contain $3$ or more edges. In one of the classical results in extremal combinatorics, Ruzsa and Szemer\'edi \cite{RSz} showed that $\exlin(n, C_3^3) = o(n^2)$. Recently,  Collier-Cartaino, Graber and Jiang \cite{CartainoGraberJiang} showed that $\exlin(n, C_{\ell}^r) = O(n^{1+\frac{1}{\lfloor \ell/2\rfloor}})$ for all $r, \ell \ge 3$.  See \cite{BergeC5, K2t, K2ttimmons} for a recent study of Linear Tur\'an numbers for Berge hypergraphs. 

\vspace{2mm}

For any integer $r \ge 2$,  the \emph{$r$-fan}, $F^r$, is the $r$-uniform linear hypergraph having $r+1$ edges $f_1, \ldots, f_r$ and $g$ such that $f_1, \ldots, f_r$ all contain the same vertex $v$ and $g$ intersects all $f_i$ in a vertex different from $v$. The $3$-fan is also called as \emph{sail} (which is configuration $C_{15}$ in \cite{Colbourn}); see Figure \ref{sailfigure}. The classical Tur\'an number $\ex(n, F^r)$ is determined by Mubayi and Pikhurko \cite{MubayiPikhurko}, who showed that the extremal hypergraphs are complete $r$-partite $r$-graphs with parts of almost equal size.

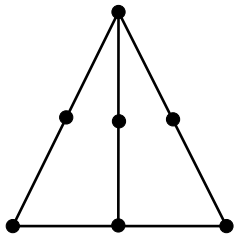
\begin{figure}[h]
\label{sailfigure}
\centering
\begin{tikzpicture}[line cap=round,line join=round,x=4cm,y=4cm]
\draw [line width=1.pt] (-7.0858666666667,5.776809384679896)-- (-7.437152447012416,5.064717008355804);
\draw [line width=1.pt] (-6.727104903779861,5.064717008355804)-- (-7.0858666666667,5.776809384679896);
\draw [line width=1.pt] (-6.727104903779861,5.064717008355804)-- (-7.437152447012416,5.064717008355804);
\draw [line width=1.pt] (-7.086458233586581,5.066881787451026)-- (-7.0858666666667,5.776809384679896);
\begin{scriptsize}
\draw [fill=black] (-7.437152447012416,5.064717008355804) circle (2.5pt);
\draw [fill=black] (-7.0858666666667,5.776809384679896) circle (2.5pt);
\draw [fill=black] (-6.727104903779861,5.064717008355804) circle (2.5pt);
\draw [fill=black] (-7.086458233586581,5.066881787451026) circle (2.5pt);
\draw [fill=black] (-7.259640561204277,5.426235117257742) circle (2.5pt);
\draw [fill=black] (-6.904616789587999,5.419740779972079) circle (2.5pt);
\draw [fill=black] (-7.08429345449136,5.4132464426864155) circle (2.5pt);
\end{scriptsize}
\end{tikzpicture}
\caption{Sail (or $F^3$)}
\end{figure}

F\"{u}redi and Gy\'arf\'as \cite{FurediGyarfas} studied the linear Tur\'an number of $F^r$, and showed that $\exlin(n, F^r) \le n^2/r^2$. Moreover, they showed that $\abs{E(H)} = n^2/r^2$ holds if and only if $n \equiv 0 \pmod{r}$ and $H$ is a \emph{Transversal design} $T(n, r)$ on $n$ vertices with $r$ groups -- where $T(n,r)$ is defined as an $r$-partite hypergraph with groups of equal size (thus $n$ is a multiple of $r$) such that each pair of vertices from different groups is contained in exactly one hyperedge. It is well-known that transversal designs exist for all $n > n_0(r)$ if $r$ divides $n$. A \emph{truncated design} is obtained from a transversal design by removing one vertex, and all the hyperedges containing it. 

\vspace{2mm}

F\"{u}redi and Gy\'arf\'as \cite{FurediGyarfas} determined the exact value of $\exlin(n, F^r)$ only in the cases when $n \equiv 0 \pmod{r}$ or $n \equiv -1 \pmod{r}$. They mentioned that determining the exact value of $\exlin(n, F^r)$ for all values of $n$ seems to be a difficult problem, let alone giving a description of all extremal hypergraphs. They also remarked that the study of the remaining cases might reveal some (possibly infinitely many) exceptional extremal configurations. For triple systems, they showed that if $n = 3k$, then $\exlin(n, F^3) = k^2$, with Transversal designs $T(3k,3)$ being the extremal hypergraphs (as discussed above). In the case when $n = 3k+2$, they showed that $\exlin(n, F^3) = k^2+k$, and the only extremal hypergraphs are truncated designs obtained from a transversal design $T(3k+3, 3)$, along with three small hypergraphs. 

\vspace{2mm}

This leaves the case $n = 3k+1$  open. Addressing the question of F\"{u}redi and Gy\'arf\'as \cite{FurediGyarfas}, we solve this remaining case, thus completing the determination of $\exlin(n, F^3)$ for all $n$. We give a characterization of all the extremal hypergraphs. Surprisingly, this case leads to a rich set of new extremal hypergraphs that are rather non-standard and seem to be very different in spirit from the extremal hypergraphs for $n = 3k$ and $n = 3k+2$. In particular, they are not derived from transversal designs.


\begin{thm}
	\label{Main_Result}
If $n=3k+1$, then $\exlin(n, F^3) = k^2+1$. The only extremal hypergraphs are given by the four constructions below.
\end{thm}








\begin{construction}
\label{generalconstr}
Let $k\ge 3$ be an integer, and let  $X=\{x_1,x_2,\ldots,x_k\}$, $Y=\{y_1,y_2,\ldots, y_k\}$, $Z=\{z_1,z_2,\ldots, z_{k-2},a,b,c\}$ be three vertex-disjoint sets. 
Let $B$ be a complete bipartite graph with parts $X$ and $Y$.

Let  $B_0=\{C_1,C_2,\ldots, C_l\}$ be a $2$-factor of $B$, where for each $i\in[l]$, $C_i$ denotes a cycle and the cycle $C_l$ is of length at least $6$. Take two disjoint edges $a'c',x'y'\in C_l$, such that $a',x'\in X$, $c',y'\in Y$ and $a'y',c'x'\notin C_l$. 
Properly $2$-color the edges of $B_0\setminus \{a'c',x'y'\}$ with colors $a$ and $c$, such that the color of the edge adjacent to $a'$ (in $B_0\setminus \{a'c',x'y'\}$) is $a$ and the color of the edge adjacent to $c'$ (in $B_0\setminus \{a'c',x'y'\}$) is $c$. Color $x'y'$ with the color $b$. 

Now we will construct the hypergraph $H_1$ as follows. For each edge $uv\in B_0\setminus \{x'y'\}$, let us add the edge $uva$ to $H_1$ if $uv$ has the color $a$; add the edge $uvb$ if $uv$ has the color $b$; add the edge $uvc$ if $uv$ has the color $c$.

Notice that $B\setminus B_0$ is $(k-2)$-regular. Decompose $B\setminus B_0$ arbitrarily into $k-2$ matchings $M_1,M_2,\ldots, M_{k-2}$, where each of the matchings is of size $k$ and for each edge $xy\in M_i$, add the edge $xyz_i$ to $H_1$ (for each $i \in \{1,2, \ldots, k-2\}$). Finally, add the two edges $a'bc$ and $c'ab$ to $H_1$.
\end{construction}

\begin{construction}
\label{seondconstr}
Let $k \ge 3$ be an integer divisible by $3$. Let $X=\{x_1,x_2,\ldots ,x_k\}$,  $Y=\{y_1,y_2,\ldots ,y_k\}$, $Z=\{z_1,z_2,\ldots ,z_{k-2},a,b,c\}$ be three vertex-disjoint sets. Take a complete bipartite graph $B$ with parts $X$ and $Y$. Let $H_2$ be the hypergraph consisting of the edges described below:


Consider any $2$-factor $B_0$ of $B$, such that $B_0$ consists of cycles of length divisible by $6$. Color edges of $B_0$ with colors $a,b,c$ such that every $3$ consecutive edges have different colors. 
For each edge $uv\in B_o$, add $uva$ to $H_2$ if $uv$ has color $a$; add $uvb$ to $H_2$ if $uv$ has color $b$; add $uvc$ to $H_2$ if $uv$ has color $c$.


Now take an arbitrary decomposition of the ($(k-2)$-regular) graph $B\setminus B_0$ into $k-2$ perfect matchings $M_1,M_2,$ $\ldots,M_{k-2}$.
For each edge $xy\in M_i$ with $i \in \{1,2, \ldots, k-2\}$, add the edge $xyz_i$ to $H_2$. Finally, add the edge $abc$ to $H_2$. 
\end{construction}

\begin{construction}\label{smallconstruction1}

Let $k=3$. Let $B$ be a graph on vertex set $\{x_1,x_2,x_3,y_1,y_2,y_3\}$  with the edge set $\{x_1x_2,x_2x_3,x_3x_1,y_1y_2,y_2y_3,y_3y_1,x_1y_1,x_2y_2,x_3y_3\}$. 
Let $Z=\{a,b,c,v\}$.

Let $M_v=\{x_1y_1,x_2y_2,x_3y_3\}$. Decompose $B\setminus M_v$ into three matchings $M_a, M_b$ and $M_c$ such that each of them is of size two. 

Let $H_3$ be the hypergraph with the following $10$ edges:
$3$ hyperedges obtained by adding the vertex $v$ to the edges of $M_v$, and the
$6$ hyperedges obtained by adding the vertex $x$ to the edges of $M_x$ (for $x \in \{a,b,c\}$), and the hyperedge $abc$.
\end{construction}

\begin{construction}\label{smallconstruction2}

Let $k=3$.
Take two disjoint sets of vertices $\{x_1,x_2,x_3,y_1,y_2,y_3\}$ and $\{v,a,b,c\}$.
Define matchings  $M_a=\{y_1y_2,x_1x_2\}$, $M_b=\{x_2x_3\}$ and $M_c=\{y_1y_3,x_1x_3\}$. Let $M_v$ be either $\{x_1y_1,x_2y_2,x_3y_3\}$, or $\{x_1y_2,x_2y_1,x_3y_3\}$ or $\{x_1y_3,x_3y_1,x_2y_2\}$.

Let $H_4$ be the hypergraph with the $10$ edges given as follows:
The hyperedges obtained by adding the vertex $w$ to the edges of $M_w$ (for $w \in \{a,b,c,v\}$), and the hyperedges $aby_3, bcy_2$.
\end{construction}
 
 
\begin{remark}
Construction \ref{generalconstr} and Construction \ref{seondconstr} both provide many non-isomorphic extremal configurations for any fixed $k > 3$.

Also note that for $k >3$ all of the extremal hypergraphs are given by Construction \ref{generalconstr} and Construction \ref{seondconstr}. For $k =3$,  however, there are two additional extremal hypergraphs -- Construction \ref{smallconstruction1} and Construction \ref{smallconstruction2}.
\end{remark}


\noindent \textit{Notation.} For a hypergraph $H$, let $d(v)$ denote degree of a vertex $v$ in $H$. The link of a vertex $x$ is denoted by $L(x)$, $L(x)=\{uv \mid uvx\in E(H)\}$. Let $\partial H = \{xy \mid \{x, y\} \subseteq e \in E(H) \}$ denote the $2$-shadow of $H$. For $v\in V(H)$, let $N(v) = \{x \mid xv \in \partial H \}$ denote the neighborhood of $v$ in $H$. Let $S(v)=V\setminus N(v)$. Note that $S(v)$ contains $v$.

For a graph $G$, we sometimes also use $G$ to denote its edge-set. For example, for a matching $M$, we denote its edge-set also by $M$. For two graphs $G, G'$ on the same vertex set, $G\setminus G'$ denotes the graph $(V(G),E(G)\setminus E(G'))$. For a graph $G$ and a vertex $x \in V(G)$, the degree of $x$ in $G$ is denoted by $d_G(x)$. 

 Theorem \ref{Main_Result} is proved in Section \ref{sec2} and the above four constructions are proven to be sail-free in Section \ref{constructionsproof}.

\section{Proof of Theorem \ref{Main_Result}}
\label{sec2}

Suppose $H$ is a sail-free linear $3$-uniform hypergraph on $n=3k+1$ vertices with at least $k^2+1$ edges. We will show that $H$ has exactly $k^2+1$ hyperedges, and characterize all such hypergraphs. 

Notice that for any vertex $v$ of $H$, there is no hyperedge of $H$ contained in $N(v)$, because otherwise, we have a sail in $H$, a contradiction. We will use this fact throughout the proof.

\begin{claim}
\label{maxdegreeisk}
The maximum degree in $H$ is equal to $k$.
\end{claim}
\begin{proof}
Let us denote the maximum degree of $H$ by $\Delta$, and let $v$ be a vertex with degree $\Delta$. We will show $\Delta = k$, which would prove the claim. Firstly, notice that $\Delta \geq k$. Indeed, otherwise, the total number of edges in $H$ is at most $(3k+1)(k-1)/3<k^2+1$, contradicting our assumption. 
Since there is no hyperedge of $H$ which is contained in $N(v)$, every hyperedge of $H$ contains a vertex from $V(H)\setminus N(v)$. Thus $\abs{E(H)}\leq (3k+1-\abs{N(v)})\Delta$, so we have $$k^2+1\leq \abs{E(H)} \leq (3k+1-\abs{N(v)})\Delta = (3k+1-2 \Delta)\Delta.$$ It is easy to check that $(3k+1-2\Delta)\Delta$ is a decreasing function in $\Delta$ for $\Delta \geq k$, and the inequality is not satisfied for $\Delta = k+1$.  Therefore, $\Delta\leq k$, but as we noted before $\Delta \ge k$. This implies $\Delta=k$, as desired.
\end{proof}

\begin{definition}
Let $E_1(v)$ be the set of hyperedges of $H$ which have $2$ vertices in $N(v)$ and $1$ vertex in $S(v)$. Let $E_2(v)$ denote the set of hyperedges of $H$ with $1$ vertex in $N(v)$ and $2$ vertices in $S(v)$ and let $E_3(v)$ be the set of hyperedges contained in the set $S(v)$.
\end{definition}
Notice that, since there is no hyperedge of $H$ which is contained in $N(v)$ we have $|E(H)|=|E_1(v)|+|E_2(v)|+|E_3(v)|$.

Claim \ref{Deficiensy} states that for any vertex $x$, its degree $d(x) \le k$. So one may view $k - d(x)$ as the deficiency of $x$. This notion is crucial to our proof, and is made precise below.

\begin{definition}[Deficiency]
For a set $S\subseteq V$, we define the deficiency of the set $S$ as $\Def(S)=\sum_{x\in S}(k-d(x))=|S|\cdot k-\sum_{x\in S}d(x)$. 
\end{definition}

\begin{lemma}
    \label{Deficiensy}
The deficiency of the vertex set of $H$ is $$\Def(V(H))\leq k-3.$$ Moreover, for every $v\in V(H)$ with $d(v)=k$, we have $$\Def(S(v))\leq k-|E_2(v)|-2|E_3(v)|-1.$$
\end{lemma}

\begin{proof}
Let us first show $\Def(V(H))\leq k-3$. By our assumption that $\abs{E(H)}\geq k^2+1$ we have $\sum_{x\in V(H)}d(x)\geq 3k^2+3$, so $\Def(V(H))=|V(H)|\cdot k-\sum_{x\in V(H)}d(x)\leq 3k^2+k-3k^2-3=k-3$, proving the first part of the lemma.

Now, consider a vertex $v\in V(H)$ such that $d(v)=k$. We want to prove that $\Def(S(v))=k-|E_2(v)|-2|E_3(v)|-1$.
Since $\abs{N(v)}=2k$, we have $|S(v)|=k+1$. Therefore, $\Def(S(v))=k^2+k-\sum_{x\in S(v)}d(x)$, i.e., we have
\begin{equation}
\label{eq1}
   \sum_{x\in S(v)}d(x)= k^2+k-\Def(S(v)).
\end{equation}


Recall that \[|E(H)|=|E_1(v)|+|E_2(v)|+|E_3(v)|\] and by the definition of $E_1(v)$, $E_2(v)$ and $E_3(v)$, we have $$\sum_{x\in S(v)}d(x)=|E_1(v)|+2|E_2(v)|+3|E_3(v)|.$$
Combining the above two equations, we get
$$|E(H)|=\sum_{x\in S(v)}d(x)-|E_2(v)|-2|E_3(v)|.$$

Using the fact that $\abs{E(H)}\geq k^2+1$ we have
$$\sum_{x\in S(v)}d(x)-|E_2(v)|-2|E_3(v)|\geq k^2+1.$$
This inequality, together with \eqref{eq1}, shows that
$$k^2+k-\Def(S(v))-|E_2(v)|-2|E_3(v)|\geq k^2+1.$$
Rewriting this inequality finishes the proof of the lemma.
\end{proof}

\begin{definition}
\label{def3}
For each vertex $x\in S(v)$, let the number of hyperedges of $E_1(v)$ which are adjacent to $x$ be $d_1^v(x)$. Let the number of hyperedges in $E_2(v)$ which are adjacent to $x$ be $d_2^v(x)$ and let the number of hyperedges of $E_3(v)$ adjacent to $x$ be $d_3^v(x)$. 
\end{definition}

Even though $d_1^v(x)$, $d_2^v(x)$ and $d_3^v(x)$ depend on a choice of vertex $v$, for convenience we drop $v$ from the notation, when the choice of $v$ is clear. Observe that $d(x)=d_1(x)+d_2(x)+d_3(x)$ for each vertex $x\in S(v)$. It is worth noting, that $v$ is an isolated vertex in $(\partial H)[S(v)]$ and $d(v)=d_1(v)=k$.

We divide the rest of the proof into the two complementary cases.

\subsection{Case 1: There exists a vertex   \texorpdfstring{$v$ such that $d(v)=k$ and $E_3(v)\ne \varnothing$}{}}

Clearly in this case $k \ge 3$.

Fix the vertex $v$ mentioned in the statement of Case 1. Let $abc\in E_3(v)$. Then, by definition, $\{ a,b,c \}\subseteq S(v)$, moreover it is easy to see that $\{ a,b,c \}\subseteq S(v)\setminus\{v\}$. Recall that by Lemma \ref{Deficiensy}, the deficiency of the whole vertex set is at most $k-3$. Now we will show that in fact, the deficiency of $\{ a,b,c \}$ must be $k-3$ (so, $\Def(V(H) \setminus \{ a,b,c \}) = 0$).

\begin{claim}
    \label{linkkmatching}
We have $\Def(\{ a,b,c \})=k-3 =\Def(S(v))$. Moreover,
$E_2(v)=\emptyset$ and $E_3(v)=\{abc\}$.
\end{claim}
\begin{proof}

Let $X=\{xy\in \partial H |x\in N(v), y\in \{a,b,c\}\}$ be the set of edges of $\partial H$, between sets $N(v)$ and $\{a,b,c\}$.
It is easy to see that $\abs{X}=2(d_1(a)+d_1(b)+d_1(c))+d_2(a)+d_2(b)+d_2(c)$. Also, for each vertex $x\in N(v)$, at most two of the edges $xa,xb,xc$ can belong to $X$, otherwise there would be a hyperedge $abc$ in a $N(x)$, which means we have a sail in $H$, a contradiction. This implies that $\abs X \leq 2\abs{N(v)}=4k$.






Therefore, we have 
\begin{equation}
 \label{nabc}
    2(d_1(a)+d_1(b)+d_1(c))+d_2(a)+d_2(b)+d_2(c)\leq 4k.
\end{equation}
Rewriting, we get,
$$d_1(a)+d_1(b)+d_1(c)+d_2(a)+d_2(b)+d_2(c)\leq 2k+\frac{d_2(a)+d_2(b)+d_2(c)}{2}.$$
Notice that $d_3(a)+d_3(b)+d_3(c)\leq |E_3(v)|+2$.
Combining this inequality with the previous inequality, we get
$$d(a)+d(b)+d(c)\leq 2k+\frac{d_2(a)+d_2(b)+d_2(c)}{2}+|E_3(v)|+2.$$ Therefore, using that $\Def(\{ a,b,c \}) = 3k - d(a)-d(b)-d(c)$, we get
$$\Def(\{ a,b,c \})\geq k-\frac{d_2(a)+d_2(b)+d_2(c)}{2}-|E_3(v)|-2 \geq k-\frac{|E_2(v)|}{2}-|E_3(v)|-2.$$

By Lemma \ref{Deficiensy},
\begin{equation}
\label{long3}
    k-|E_2(v)|-2|E_3(v)|-1\geq \Def(S(v))\geq \Def(\{ a,b,c \})\geq k-\frac{|E_2(v)|}{2}-|E_3(v)|-2.
\end{equation}
Rewriting, we get,
\begin{equation}
\label{eqholds}
  \frac{|E_2(v)|}{2}+|E_3(v)| \leq 1.
\end{equation}
Since we assumed $|E_3(v)|\geq 1$,  it follows that $|E_3(v)|=1$ and $|E_2(v)|=0$. Moreover, \eqref{eqholds} has to hold with equality. This means that all of the inequalities in \eqref{long3} have to hold with equality. So  $\Def(S(v))= \Def(\{ a,b,c \})=k-3$ and $|E_3(v)|=1$, which clearly implies that $E_3(v)=\{abc\}$, finishing the proof of the claim. 
\end{proof}


For each vertex $x\in S(v)$, let us define a matching $M_x := \{ wu\in L(x)|w,u\in N(v)\}$. Let $M_{abc}$ be a graph whose edge-set is $M_a\cup M_b \cup M_c$. 
By Claim \ref{linkkmatching}, we have $\Def(\{a,b,c\})=\Def(S(v))=k-3$, so $2k+3=d(a)+d(b)+d(c)$. Moreover, Claim \ref{linkkmatching} also asserts $\abs{E_2(v)}=0$ and $E_3(v)=\{abc\}$, which implies that $2k+3=d(a)+d(b)+d(c)=d_1(a)+d_1(b)+d_1(c)+3$, so 
\begin{equation}\label{mabc}
  |M_{abc}|=d_1(a)+d_1(b)+d_1(c)=2k.  
\end{equation}

\begin{claim}
\label{3cycles}
  $M_{abc}$ is a disjoint union of cycles whose lengths are divisible by $3$. Moreover, each cycle of $M_{abc}$ is a cyclic sequence of edges of $M_a,M_b,M_c,\ldots, M_a,M_b,M_c$.
\end{claim}
\begin{proof}
 Let $x\in N(v)$, if $x\in V(M_a)\cup V(M_b)\cup V(M_c)$ then the hyperedge $abc\subseteq N(x)$, a contradiction. So the maximum degree in the graph $M_{abc}$ is $2$, but by \eqref{mabc}, $\abs{M_{abc}}=2k=\abs{N(v)}$ which means that every vertex must have degree exactly $2$. So $M_{abc}$ is disjoint union of cycles. 
 Let $u_0u_1$, $u_1u_2$ and $u_2u_3$ be three consecutive edges of any cycle in $M_{abc}$ ($u_0$ might be the same as $u_3$). We claim that the matchings $M_a$, $M_b$ and $M_c$ each contain exactly one of the edges $u_0u_1$, $u_1u_2$, $u_2u_3$. It suffices to show that two of the edges $u_0u_1$, $u_1u_2$, $u_2u_3$ cannot be in the same matching. Let us assume for a contradiction this is not true. Without loss of generality, we can assume that two of these edges are in $M_a$ and because of the linearity of $H$, these two edges have to be $u_0u_1$ and $u_2u_3$. If $u_1u_2\in M_b$ then it is easy to see that the hyperedge $u_1u_2b\subseteq N(a)$ and if $u_1u_2\in M_c$ then $u_1u_2c\subseteq N(a)$, a contradiction. So for every $3$ consecutive edges in any cycle of $M_{abc}$, we must have one in $M_a$, one in $M_b$ and one in $M_c$, which trivially implies the claim.
\end{proof}

Let $B=(\partial H)[N(v)]$. Since there is no hyperedge of $H$ completely contained in $N(v)$ we have $B=\cup_{x\in S(v)}M_x$. 
By Claim \ref{linkkmatching}, $\Def(S(v)\setminus \{ a,b,c \})=0$, so for each $x\in S(v)\setminus \{a,b,c \}$ we have that $M_x=L(x)$ is a matching of size $k$ in $B\setminus M_{abc}$. Therefore $B\setminus M_{abc}$ is a union of $k-2$ matchings of size $k$. Moreover, by Claim \ref{3cycles}, $M_{abc}$ is $2$-regular graph, so $B$ is a $k$-regular graph.

We separate the cases depending on whether $B$ has a triangle or not. 

\begin{claim}
\label{notriangle0}
    If $xyz$ is a triangle in $B$, then $xyz$ is a triangle in $M_{abc}$.
\end{claim}
\begin{proof}
Let us assume for a contradiction that $\{xy,yz,zx\}\not \subseteq M_{abc}$. Without loss of generality we can assume that $xy\notin M_{abc}$.
There exists a vertex $w\in S(v)\setminus \{ a,b,c \}$ such that $wxy\in E(H)$. Since $V(M_w)=N(v)$, we have $wz\in \partial H$. So the hyperedge $xyw$ is inside $N(z)$, a contradiction. 
\end{proof}

Suppose there is a triangle $xyz$ in $B$. By Claim \ref{notriangle0}, $xyz$ is a triangle in $M_{abc}$, so without loss of generality we may assume that $xy\in M_a,yz\in M_b,zx\in M_c$. Recall that by Claim \ref{3cycles}, $M_{abc}$ is a vertex-disjoint union of cycles, so $xyz$ is one of those cycles. In particular, $xyz$ is the only triangle of $M_{abc}$ containing the edges $xy, yz$ or $zx$. If there is a triangle in $B$ sharing an edge with $xyz$, it would have to be in $M_{abc}$ as well, by Claim \ref{notriangle0}, which is impossible. So $xyz$ is the only triangle of $B$ containing the edges $xy, yz$ or $zx$. Therefore, for each vertex $u\in N(v)\setminus \{x,y,z\}$ at most $1$ of the edges $ux,uy,uz$ belongs to $B$. Therefore the sum of degrees of $x,y$ and $z$ in the graph $B$ is at most $2k-3+6=2k+3$ (where the $6$ comes from the edges $xy, yz, zx$). But $B$ is a $k$-regular graph, so $2k+3\geq3k$, i.e., $k\leq 3$, so $k=3$. 
In this case, $B$ is a $3$-regular graph on $6$ vertices containing a triangle. Then one can easily check that $B$ must be a graph consisting of two triangles with a matching between them; let $B = \{x_1x_2,x_2x_3,x_3x_1,y_1y_2,y_2y_3,y_3y_1,x_1y_1,x_2y_2,x_3y_3\}$ where  $N(v)=\{x_1,x_2,x_3,y_1,y_2,y_3\}$. By Claim \ref{notriangle0}, both triangles of $B$ are contained in $M_a \cup M_b \cup M_c$, so it is easy to see that $H$ has to be derived by means of Construction \ref{smallconstruction1}.

Now suppose $B$ is triangle-free. Let $xy\in B$ and let $X$ and $Y$ be the neighborhoods of $x$ and $y$ in $B$, respectively. Since $B$ is triangle-free and $k$-regular, $X$ and $Y$ are vertex-disjoint independent sets of size $k$. But as $B$ has $2k$ vertices, this means that $B$ is bipartite with parts $X$ and $Y$. Moreover, as $B$ is $k$-regular, it must be a complete bipartite graph $K_{k,k}$. So using Claim \ref{3cycles} it is easy to see that the hypergraph $H$ can be constructed by the means of Construction \ref{seondconstr}, where $B_0=M_{abc}$ and $\{M_1,M_2,\ldots M_{k-2}\}=\{M_x\mid x\in S(v)\setminus\{a,b,c\}\}$.

\subsection{Case 2: \texorpdfstring{For every vertex $v$ such that $d(v)=k$, $E_3(v)=\varnothing$}{}}
Among the vertices with $d(v) = k$ (which exist by Claim \ref{maxdegreeisk}), let us fix a vertex $v \in V(H)$ such that $\Def(S(v))$ is maximal.

\begin{claim}
    \label{smallclaim}
We have $|E_2(v)|\ge2$.
\end{claim}
\begin{proof}
Since there is no hyperedge of $H$ contained in $N(v)$ and $E_3(v)=\emptyset$, $E(H)=E_1(v)\cup E_2(v)$. Notice that 
$\sum_{x\in N(v)}d(x) = 2\abs{E_1(v)}+\abs{E_2(v)}= 2\abs{E(H)} - \abs{E_2(v)}$. So we have
$$2k^2=k\cdot \abs{N(v)}\geq \sum_{x\in N(v)}d(x)=2\abs{E(H)}-\abs{E_2(v)}\geq 2(k^2+1)-\abs{E_2(v)}.$$
Therefore $|E_2(v)|\geq 2$, proving the claim.

\end{proof}

Like before, let $d_1(x)=d_1^v(x)$, $d_2(x)=d_2^v(x)$ and $d_3(x)=d_3^v(x)$.
\begin{definition}
For $x\in S(v)$, let $M_x=\{yz|xyz\in E_1(v)\}$(Notice that $|M_x|=d_1(x)$). 
\end{definition}

A star is a set of edges having a common vertex. 
We plan to show that the edge set of $(\partial H)[S(v)]$ is a star.
First we will prove a small claim, which will be used throughout the proof.

\begin{claim}
    \label{newsmallclaim}
    Let $x,y\in S(v)$ and $xy \in E(\partial H)$. There is no path of  length $3$ in the graph $M_x\cup M_y$ and $d_1(x)+d_1(y)\leq \lfloor{\frac{4k-2}3}\rfloor$.
\end{claim}

\begin{proof}
Let us assume $M_x\cup M_y$ contains a $3$-path, i.e., there are vertices $u_0,u_1,u_2,u_3\in N(v)$  such that $u_0u_1,u_1u_2,u_2u_3\in E(\partial H)$. Without loss of generality, let us assume $u_1u_2\in M_x$, i.e., $xu_1u_2\in E(H)$. By linearity of $H$, $u_0u_1,u_2u_3\in M_y$, so $u_1y, u_2y\in E(\partial H)$. Moreover, we assumed $xy \in E(\partial H)$, so the hyperedge $xu_1u_2 \in E(H)$  is in the neighborhood of $y$,  giving us a sail in $H$, a contradiction, proving the  first part of the claim.

 Now, let $z$ be a vertex such that $xyz\in E(H)$. Then as $E_3(v) =\emptyset$, $z \in N(v)$. By linearity of $H$ it is easy to see that $z\notin V(M_x)\cup V(M_y)$, so combining this with the fact that $M_x \cup M_y$ does not contain a $3$-path, we have $d_1(x)+d_1(y)=\abs{M_x \cup M_y}\leq \lfloor{\frac{2}3 (2k-1)}\rfloor$.
\end{proof}

\begin{claim}
    \label{2path}
Let $a,b,c\in S(v)$ and $ab,bc\in E(\partial{H})$. Then $ac\notin E(\partial H)$ and $d_2(a)+d_2(b)+d_2(c)=\abs{E_2(v)}+2$. Moreover, $\Def(S(v))=\Def(\{a,b,c\})=k-\abs{E_2(v)}-1$.
\end{claim}

\begin{proof}
Let $M_{abc}=M_a\cup M_b \cup M_c$. For a vertex $x\in V(M_{abc})$, let $d'(x)$ denote the degree of $x$ in $M_{abc}$.
Since $M_a,M_b$ and $M_c$ are matchings,  $d'(x)\leq 3$ for every $x\in V(M_{abc})$. Since $E_3(v)=\emptyset$ there exist vertices $a',c'\in N(v)$ such that $abc',bca'\in E(H)$.

Let $x\in N(v)\setminus \{a',c'\}$ with $d'(x)=3$ and let $a_0,b_0,c_0\in N(v)$ be vertices such that $xa_0,xb_0,xc_0\in M_{abc}$. In particular, let $xa_0\in M_a,xb_0\in M_b$ and $xc_0\in M_c$.
We claim that $d'(b_0)=1$. 
Let us assume that $d'(b_0)>1$, then there is a vertex $y\not=x$ such that $b_0y\in M_{abc}$. Since $M_b$ is a matching, $b_0y\notin M_b$. If $b_0y\in M_a$, then $a_0xb_0y$ is a path of length $3$ in $M_a\cup M_b$ and if $b_0y\in M_c$, then $c_0xb_0y$ is a path of length $3$ in $M_b\cup M_c$, contradicting Claim \ref{newsmallclaim}. Therefore $d'(b_0)=1$. Moreover $bxb_0,abc'$ and $bca'\in E(H)$, so by linearity $b_0\notin \{a',c'\}$.
So for every $x\in N(v)\setminus \{a',c'\}$ with $d'(x)=3$ there exists a vertex $b_0\in N(v)\setminus \{a',c'\}$, which is adjacent to $x$ with $d'(b_0)=1$. Therefore we have, 
\begin{equation}
    \label{deg3deg1}
    \abs{\{x\in N(v)\setminus\{a',c'\}\mid d'(x)=3\}}\le \abs{\{x\in N(v)\setminus\{a',c'\}\mid d'(x)=1\}}.
\end{equation}

Now we will prove that $ac\notin E(\partial H)$. Suppose otherwise. Then there exists $b'\in N(v)$ such that $acb'\in E(H)$.
By linearity of $H$, $b'\notin V(M_a)\cup V(M_c$), and if $b'\in V(M_b)$, then the hyperedge $acb'$ would be in the neighborhood of $b$, resulting  in the existence of a sail in $H$. Therefore $b'\notin V(M_{abc})$. By symmetry we have $a',b',c'\notin V(M_{abc})$, so $\abs{V(M_{abc})}\leq \abs{N(v)}-3=2k-3$. By \eqref{deg3deg1}, the average degree in $M_{abc}$ is at most $2$, so $\abs{M_{abc}}\leq\abs{V(M_{abc})}\leq2k-3$.
By the definition of $M_{abc}$, $d_1(a)+d_1(b)+d_1(c)=|M_{abc}|\leq 2k-3$. Note that $d_2(a)+d_2(b)+d_2(c)\leq |E_2(v)|+3$, as the three hyperedges $abc',acb',a'bc\in E_2(v)$ are double counted by the sum $d_2(a)+d_2(b)+d_2(c)$. Adding up the two previous inequalities, we have $$d(a)+d(b)+d(c)\leq 2k-3+|E_2(v)|+3=2k+|E_2(v)|.$$  Then $$\Def(S(v))\geq \Def(\{a,b,c\})\geq k-|E_2(v)|,$$ contradicting Lemma \ref{Deficiensy} and proving that $ac\notin E(\partial H)$.

Now we will prove that $d_1(a)+d_1(b)+d_1(c)=2k-1$ by showing $\sum_{x\in N(v)}d'(x)\leq 4k-2$. 
By the linearity of $H$, $a'\notin V(M_{b})\cup V(M_{c})$, so $d'(a')\leq 1$. Similarly, we can show $d'(c')\leq 1$. 
Using \eqref{deg3deg1}, it follows that $\sum_{x\in N(v)\setminus \{a',c'\}}d'(x)\leq 4k-4$, so $\sum_{x\in N(v)}d'(x)\leq 4k-2$ which implies that $|M_{abc}|\leq 2k-1$. In other words,

\begin{equation}
    \label{2k-1}
d_1(a)+d_1(b)+d_1(c)=|M_{abc}|\leq 2k-1.
\end{equation}

Since $ac\notin E(\partial H)$,
\begin{equation}
    \label{d2e2}
    d_2(a)+d_2(b)+d_2(c)\leq |E_2(v)|+2
\end{equation}
Adding up the two previous inequalities, we have $$d(a)+d(b)+d(c)\leq 2k-1+|E_2(v)|+2=|E_2(v)|+2k+1.$$
Then $\Def(\{a,b,c\})\geq k-|E_2(v)|-1$. Combining this with Lemma \ref{Deficiensy}, we  get $\Def(S(v))=\Def(\{a,b,c\})=k-|E_2(v)|-1$. Moreover equality has to hold in every inequality which was used to derive the last equation, so equality holds in \eqref{d2e2}, proving the claim.
\end{proof}

\begin{claim}
    \label{notwomatching}
    There are no distinct vertices $a,b,u,w\in S(v)$ such that $ab,uw\in \partial H$.
\end{claim}
\begin{proof}
Assume for a contradiction that $ab,uw\in \partial H$. First we will show that $(\partial H)[\{a,b,u,w\}]$ does not have a $4$-cycle.
Assume the contrary, that $(\partial H)[\{a,b,u,w\}]$ contains a $4$-cycle, then without loss of generality we may assume $bu,aw\in \partial H$. 
By Claim \ref{2path}, $d_2(a)+d_2(b)+d_2(u)=\abs{E_2(v)}+2$, and $au\notin(\partial H)[S(v)]$, which implies that every hyperedge of $E_2(v)$ contains one of the vertices $a,b$ or $u$. Similarly, $bw \notin (\partial H)[S(v)]$ and every hyperedge of $E_2(v)$ contains one of the vertices $b, u$ or $w$. Therefore $(\partial H)[S(v)] = \{ab,uw,bu,aw\}$, so $\abs{E_2(v)}=4$.
Claim \ref{2path} also asserts that
 $k-|E_2(v)|-1=\Def(S(v))=\Def(\{a,b,u\})=\Def(\{b,u,w\})=\Def(\{u,w,a\})=\Def(\{w,a,b\})$. It follows that $\Def(\{x\})=0$ for each $x\in S(v)$. Then we have $k-5=k-|E_2(v)|-1=\Def(S(v))=0$, i.e., $k=5$. In this case, $\abs{E(H)}=\sum_{x\in S(v)}d(x)-\abs{E_2(v)}=6\cdot 5-4=26$, so $\Def(V(H))=16\cdot 5-3\cdot 26=2$. Therefore, there exists a vertex $x\in V(H)$ such that $d(x)\le 4$. 
  Since $\abs{V(H)\setminus (N(x)\cup \{x\})}\ge 16-8-1=7$, there exists a vertex $v'\in V(H)\setminus (N(x)\cup\{x\})$ such that $d(v')=5=k$, otherwise $\Def(V(H))$ must be at least $7$ but $\Def(V(H)) = 2$, a contradiction. It is easy to see that $x\in S(v')$ therefore $\Def(S(v))<1\le \Def(\{x\})\le \Def(S(v'))$, which contradicts our assumption that $v$ was chosen such that $\Def(S(v))$ is maximal among vertices of degree $k$.


 Therefore it is impossible that $bu,aw\in \partial H$ and similarly it is impossible that $bw,au\in \partial H$, so $(\partial H)[\{a,b,u,w\}]$ does not contain a $4$-cycle and by Claim \ref{2path}, it  does not contain a triangle. Hence, $d_2(a)+d_2(b)+d_2(u)+d_2(w)\le \abs{E_2(v)}+3$.

By Claim \ref{newsmallclaim}, $d_1(a)+d_1(b)+d_1(u)+d_1(w)\le 2\lfloor{\frac{4k-2}3}\rfloor$,
so $d(a)+d(b)+d(u)+d(w)\le 2\lfloor{\frac{4k-2}3}\rfloor+\abs{E_2(v)}+3$. By Lemma \ref{Deficiensy}, $k-|E_2(v)|-1\ge\Def(S(v))\geq 4k-(d(a)+d(b)+d(u)+d(w))$.  Therefore, 
$$3k+|E_2(v)|+1\le d(a)+d(b)+d(u)+d(w)\le 2\left \lfloor{\frac{4k-2}3}\right \rfloor+\abs{E_2(v)}+3,$$
so $3k\le2\lfloor{\frac{4k-2}3}\rfloor +2$, i.e., $k\le2\lfloor{\frac{k-2}3}\rfloor +2$, 
a contradiction, because $k \ge 3$, proving the claim.
\end{proof}

By Claim \ref{2path} and Claim \ref{notwomatching}, $(\partial H)[S(v)]$ is triangle-free and does not contain a matching of size $2$, therefore its edge-set forms a star. 


\begin{claim}
\label{E2=2}
$|E_2(v)|=2$.
\end{claim}
\begin{proof}
Assume for a contradiction that $|E_2(v)|\geq 3$.
Let $E_2(v)=\{oa_1b_1, oa_2b_2, ... oa_pb_p\}$ where $p \geq 3$, $o\in S(v)$, and for each $i\in[p]$ we have $a_i\in S(v)$ and $b_i\in N(v)$. Note that $(\partial H)[S(v)]$ is a star with the edge-set $\{oa_i|1\leq i\leq p\}$. Recall that $v$ is an isolated vertex in $(\partial H)[S(v)]$, so $k=\abs{S(v)\setminus\{v\}}\ge p+1$.

By Claim \ref{2path}, for any $\{i,j\}\subset [p]$, $k-p-1=\Def(S(v))=\Def(\{a_i,o,a_j\})$ and since $p>2$, this means that we have $\Def(\{a_i\})=0$ for every $i\in[p]$. Therefore, 
\begin{equation}
\label{starcenter}
    k-p-1=\Def(S(v))=\Def(\{o\}).
\end{equation}
 So $d(o) = p+1$. Since $d_2(o)=\abs{E_2(v)}=p$, we get $d_1(o)=1$. Let $ouw$ be the only edge of $E_1(v)$ which is incident to $o$.
Let $Q=\{h\in E(H)\mid u\in h \text{ or } w\in h\}$. For $x\in \{o,a_1,a_2,\ldots a_p\}$, there is at most one hyperedge $h\in Q$, such that $x\in h$, otherwise the hyperedge $ouw$ would be in $N(x)$, giving us a sail in $H$. Since $H$ is linear, for any $x\in V(H)$ there are at most two hyperedges of $Q$ containing $x$ and since there is no hyperedge contained in $N(v)$, for every $h\in Q$ there exists a vertex $x\in S(v)$ such that $x\in h$. Therefore,
$$\abs{Q}\le\abs{\{o,a_1,a_2,\ldots, a_p\}}+ 2\abs{S(v)\setminus \{o,a_1,a_2,\ldots, a_p\}}=p+1+2(k-p)=2k+1-p.$$ 
Noting that $\abs Q = d(u) +d (w) -1$, we have, $ d(o)+(d(u)+d(w))=d(o)+(\abs{Q}+1)\le p+1+(2k+1-p+1)=2k+3 $
Thus, $\Def(\{o,u,w\}) \ge k-3$. Combining this with Lemma \ref{Deficiensy} we get,  $\Def(V(H))=\Def(\{o,u,w\})=k-3$. So for any $x\in V(H)\setminus\{o,u,w\}$, we have $d(x)=k$ and 
\begin{equation}
\label{sumdouw}
   d(o)+d(u)+d(w)=2k+3.
\end{equation}

We claim that $d(u)=d(w)=k$. Assume otherwise. Without loss of generality, we may assume $d(u)\not=k$, i.e., $d(u)\le k-1$. Since $d_1(o)=1$ and $d_2(o)=p$, $\abs{N(o)\cap N(v)}=p+2\le k+1$.
Since $ d(u)\le k-1$ it is easy to see that $\abs{N(u)\cap N(v)}\le k-1$. Notice that $w\in (N(o)\cap N(v))\cap (N(u)\cap N(v))$ therefore $\abs{(N(o)\cap N(v))\cap (N(u)\cap N(v))}\ge 1$, so we have,
$$\abs{(N(o)\cap N(v))\cup (N(u)\cap N(v))}\le \abs{N(o)\cap N(v)}+\abs{N(u)\cap N(v)}-1\le 2k-1.$$
Therefore there exists a vertex $x\in N(v)\setminus (N(o)\cup N(u))$ and since $x\notin N(o)\cup N(u)\supseteq \{o,u,w\}$, we have $d(x)=k$. By the definition of $x$ it is clear that $o,u\in S(x)$, so $\Def(S(x))\ge \Def(\{o,u\})> \Def(\{o\})=\Def(S(v))$, which contradicts our assumption that $\Def(S(v))$ was maximal. So our assumption that $d(u)<k$ is wrong, proving that $d(u)=d(w)=k$, so by \eqref{sumdouw}
we have $d(o)=3$. Then $p=d_2(o)=2$, contradicting our assumption that $p \ge 3$. Therefore $\abs{E_2(v)}=2$.
\end{proof}



Having established Claim \ref{E2=2}, let us suppose $E_2(v)=\{abc',cba'\}$ where $a,b,c\in S(v)$ and $a',c'\in N(v)$.
By Claim \ref{2path} and Lemma \ref{Deficiensy}, we have $\Def(\{a,b,c\})= k - \abs{E_2(v)} -1 = k-3 \ge \Def (V(H))$. 
But on the other hand, trivially, $\Def(\{a,b,c\}) \le \Def(V(H))$. Therefore,  $\Def (V(H))=\Def(\{a,b,c\})=k-3$, so $d(a)+d(b)+d(c)=3k - (k-3) = 2k+3$ and $d(x)=k$ for every $x\in V(H)\setminus\{a,b,c\}$. Notice that $d_2(a)+d_2(b)+d_2(c)=4$, so $d_1(a)+d_1(b)+d_1(c)=2k-1$. 

Recall, that for each $x\in S(v)$. Let $M_x=\{yz \mid xyz \in H, y, z \in N(v) \}$, and let $M_{abc}=M_a\cup M_b\cup M_c$. 
Then by the discussion in the previous paragraph, we have $\abs{M_x}=k$ for every $x\in S(v)\setminus\{a,b,c\}$ (i.e., each $M_x$ is a perfect matching) and 
\begin{equation}
\label{mabc2}
 \abs{M_{abc}}= \abs{M_a} + \abs{M_b} + \abs{M_c} = 2k-1.   
\end{equation}

Let $G=(\partial H)[N(v)]$. Notice that for every $x\in N(v)$, $d_G(x)\le d(x)\le k$ and $\abs{E(G)}=\abs{S(v)\setminus\{a,b,c\}}k+ \abs{M_{abc}} = (k-2)k + (2k-1) = k^2-1$, so precisely two vertices of $G$ have degree $k-1$ and the rest of the vertices in $G$ have degree $k$.  

Since $abc',bca'\in E(H)$, by the linearity of $H$, $a'\notin V(M_c)\cup V(M_b)$ and $c'\notin V(M_a)\cup V(M_b)$. Therefore, $d_G(a')=\abs{\{x\mid a'\in V(M_x)\}}\le k-1$ and $d_G(c')=\abs{\{x\mid c'\in V(M_x)\}}\le k-1$. It follows that the only two vertices of degree $k-1$ in $G$ are $a'$ and $c'$. This implies that $a'\in V(M_x)$ for every $x\in S(v)\setminus \{b,c\}$. In particular, $a' \in V(M_a)$. 
Notice that $a'c'\notin G$, otherwise the hyperedge $abc'$ is in $N(a')$, which is a contradiction. In summary, we proved that $a'c'\notin G$ and $d_G(a')=d_G(c')=k-1$, so $G\cup \{a'c'\}$ is a $k$-regular graph. Also, since $G\setminus M_{abc}$ is a $(k-2)$-regular graph (being the union of $k-2$ perfect matchings $M_x$ with $x \in S(v) \setminus \{a,b,c\}$), it follows that $M_{abc}\cup \{a'c'\}$ is a $2$-regular graph.



\begin{claim}
    \label{triangleinMabc}
    Suppose $u_1u_2u_3$ is a triangle in $G$. Then $u_1u_2u_3$ is a triangle in $M_{abc}$ and $u_1,u_2,u_3\in V(G)\setminus \{a',c'\}$.
\end{claim}
 \begin{proof}
 Suppose for a contradiction that $u_1u_2\notin M_{abc}$. Then there exists $x\in S(v)\setminus \{a,b,c\}$ such that $u_1u_2x\in E(H)$. Moreover, $\abs{M_x}=k$ so $V(M_x)=N(v)$ contains the vertex $u_3$. Therefore the hyperedge $u_1u_2x$ is contained in $N(u_3)$, a contradiction. 
 
 By the same argument, we have  $u_2u_3, u_1u_3 \in M_{abc}$, proving that $u_1u_2u_3$ is a triangle in $G$. It remains to show $u_1, u_2, u_3 \in V(G) \setminus \{a', c'\}$. Indeed, recall that degrees of $a'$ and $c'$ are $1$ in $M_{abc}$, proving the claim. 
 \end{proof}
 
We will distinguish two cases, depending on whether there is a triangle in $G$.  First, let us assume that there is a triangle $x_1x_2x_3$ in $G$. Then by Claim \ref{triangleinMabc}, $x_1x_2x_3$ is a triangle in $M_{abc}$ as well. Since degrees in $M_{abc}$ are at most $2$, each of the edges $x_1x_2,x_2x_3,x_1x_3$ is not contained in any other triangle of $M_{abc}$ besides $x_1x_2x_3$. Hence by Claim \ref{triangleinMabc}, each of the edges $x_1x_2,x_2x_3,x_1x_3$ is not contained in any other triangle of $G$. So for each vertex $w\in V(G)\setminus \{x_1,x_2,x_3\}$, at most one of the edges $w_1x_1,wx_2,wx_3$ are in $G$. Therefore, $d_G(x_1)+d_G(x_2)+d_G(x_3)\le 6+2k-3=2k+3$. But by Claim \ref{triangleinMabc}, $a',c'\notin \{x_1,x_2,x_3\}$, so $d_G(x_1)=d_G(x_2)=d_G(x_3)=k$, implying that $3k\le 2k+3$, i.e., $k=3$.
In this case, since $M_{abc}\cup\{a'c'\}$ is a $2$-regular graph containing the triangle $x_1x_2x_3$, $M_{abc}\cup\{a'c'\}$ is the vertex-disjoint union of two triangles $x_1x_2x_3$ and $y_1a'c'$ for some $y_1\in N(v)$. It is easy to see that $y_1a'\in M_a$ and $y_1c'\in M_b$, and without loss of generality we may assume $M_a=\{x_1x_2,y_1a'\}$, $M_c=\{x_1x_3,y_1c'\}$ and $M_b= \{x_2x_3\}$. This implies that $M_v$ is a matching between the sets $\{x_1,x_2,x_3\}$ and $\{y_1,a',c'\}$. Notice that if $x_2c'$ or $x_3a'$ are edges of $\partial H$, then $abc'$ is in $N(x_2)$ or $bca'\in N(x_3)$, a contradiction. So $x_2c',x_3a'\notin M_v$. Thus it is easy to see that $M_a,M_b,M_c$ and $M_v$ are the same as described in Construction \ref{smallconstruction2}, with $y_2=a'$ and $y_3=c'$. So $H$ can be constructed by means of Construction \ref{smallconstruction2}.

Now suppose $G$ is triangle-free. 
 Since $\abs{M_v}=k\ge 3$ there is an edge $uw\in M_v$ such that $u,w\notin \{a',c'\}$. Let $X$ and $Y$ be the sets of neighbours of $u$ and $w$ in the graph $G$, respectively. Since $G$ is triangle-free, $X$ and $Y$ are disjoint independent sets of $G$ and since $u,w\notin \{a',c'\}$ we have $d_G(u)=d_G(w)=k$, therefore $\abs{X}=\abs{Y}=k$. So $G$ is a balanced bipartite graph with parts $X$ and $Y$. We claim that $G=K_{k,k}\setminus \{a'c'\}$. As $G \cup \{a'c'\}$ is a $k$-regular graph, it suffices to prove that $a'$ and $c'$ are in different parts of $G$. Let us assume the contrary. Then without loss of generality $a',c'\in X$. For any vertex $y\in Y$ since $d_G(y)=k$, $a'y,c'y\in G$. Therefore $d_G(a')=d_G(c')=k$, a contradiction. Hence $G=K_{k,k}\setminus \{a'c'\}$ with parts $X, Y$ and $a'\in X, c'\in Y$.
 
 If there is a vertex $x\in V(M_a)\cap V(M_b)\cap X$, then clearly $x\not=a'$ so $xc'\in G\subseteq \partial H$, therefore the hyperedge $abc'$ is in $N(x)$, a contradiction. So $V(M_a)\cap V(M_b)\cap X=\emptyset$, and by symmetry it then follows that
\begin{equation}
    \label{XMB}
    V(M_a)\cap V(M_b)\cap X=V(M_b)\cap V(M_c)\cap Y=\emptyset.
\end{equation}

Let $B_0 = M_{abc}\cup\{a'c'\}$. As $B_0$ is $2$-regular, it can be decomposed into cycles $C_1,C_2,\ldots, C_l$. We may assume $a'c'\in C_l$.  We will show that $M_b\subset C_l$. Clearly $\abs{M_a}+\abs{M_c}=d(a)-1+d(c)-1\le 2k-2$. So by \eqref{mabc2}, we have $\abs{M_b} \ge 1$. 

Without loss of generality, let us assume $\abs{M_a}\ge \abs{M_c}$. Then combining this with \eqref{mabc2}, we have $\abs{M_a}\ge k-(1+\abs{M_b})/2$. Moreover, by \eqref{XMB}, we know that $M_a$ and $M_b$ are vertex-disjoint inside $X$, so $\abs{M_a}+\abs{M_b}\le \abs X = k$. Therefore combining the previous two inequalities, we get $\abs{M_b}\le(1+\abs{M_b})/2$, i.e., $\abs{M_b} \le 1$. But by the discussion in the previous paragraph, $\abs{M_b} \ge 1$, so $\abs{M_b}=1$. By Claim \ref{newsmallclaim}, both $M_a\cup M_b$ and $M_c\cup M_b$ do not contain a $3$-path, so it is easy to see that in each $C_i$ for $i\in[l-1]$ there are equal number of edges from $M_a$ and $M_c$, therefore $\abs{C_i\setminus M_b}$ is even for each $i\in[l-1]$. Then as the cycles $C_i$ are all of even length, and $M_b$ contains exactly one edge $e$, $e$ cannot be contained in any $C_i$ for $i \in [l-1]$. Therefore $M_b\subset C_l$.
Let $\{x'y'\}=M_b$ where $x'\in X$ and $y'\in Y$. If $a'y'\in C_l$, then $a'y'\in M_a$ so the hyperedge $aa'y'$ is in $N(b)$, a contradiction. Similarly $c'x'\notin M_c$, therefore we have $a'y',c'x'\notin C_l$.

Now it is easy to see that the hypergraph $H$ can be obtained by means of Construction \ref{generalconstr}, where $B=G\cup\{a'c'\}$ and $\{M_1,M_2,\ldots M_{k-2}\}=\{M_x\mid x\in S(v)\setminus\{a,b,c\}\}$ and for each $e\in B_0$, the color of $e$ is $w$ if $e\in M_w$ with $w\in \{a,b,c\}$.

\section{Proofs that the constructions are sail-free}
\label{constructionsproof}

\subsection{Proof that Construction \ref{generalconstr} is sail-free}
Let $H_1$ be the hypergraph obtained by Construction \ref{generalconstr}. For each $w\in\{a,b,c\}$, let $M_w$ be the set of edges of $B_0$ having the color $w$.

\begin{claim}
\label{generalconstrclaim}
  $x'\notin V(M_a)$ and $y'\notin V(M_c)$.
\end{claim}
\begin{proof}
Notice that $C_l\setminus\{a'c'\}$ is a properly colored path. Let $e_1,e_2,\ldots, e_p$ be the edges along this path from $a'$ to $c'$ (i.e., $a' \in e_1$ and $c' \in e_p$). By definition,  $e_1\in M_a$ and $e_p \in M_c$. Let $e_j=x'y'\in M_b$. 

First let us assume $j$ is even. Then $e_1, e_3, \ldots, e_{j-1}\in M_a$ and $e_p, e_{p-2}, \ldots, e_{j+1}\in M_c$. Notice that for all even $i\in[t]$, $e_{i-1}\cap\ e_{i}\subseteq Y$  and $e_{i}\cap\ e_{i+1} \subseteq X$, so in particular $e_{j-1}\cap e_{j} \subset Y$ and $e_{j}\cap e_{j+1}\subseteq X$, i.e., $e_{j-1}\cap e_{j}=\{y'\}$ and $e_{j}\cap e_{j+1}= \{x'\}$. So since $e_{j-1}\in M_a$ and $e_{j+1}\in M_c$ we have $y'\in V(M_a)$ and $x'\in V(M_c)$.
If $j$ is odd, by a similar argument we get $e_{j-1}\in M_c$, $e_{j+1}\in M_a$, $e_{j-1}\cap e_{j}=\{x'\}$ and $e_{j+1}\cap e_{j}=\{y'\}$ so $x'\in V(M_c)$ and $y'\in V(M_a)$. 

In summary, we have $x'\in V(M_c)$ and $y'\in V(M_a)$. It follows that since $x'y'\in M_b$, we have $x'\in V(M_b)\cap V(M_c)$ and $y'\in V(M_b)\cap V(M_a)$. Recall that $M_a\cup M_b\cup M_c\subseteq B_0$ and $B_0$ is a $2$-regular graph, so $V(M_a)\cap V(M_b)\cap V(M_c)=\emptyset$. Therefore  $x'\notin V(M_a)$ and $y'\notin V(M_c)$.
\end{proof}

Let us assume for a contradiction that there exists a sail in $H_1$. Let $w\in V(H_1)$ and $h\in E(H_1)$ such that $h\subseteq N(w)$.
Then notice that either $\abs{h\cap V(B)}=2$ and when $\abs{h\cap V(B)}=1$. 

First consider the case when $\abs{h\cap V(B)}=2$. Then there exist vertices $x\in X,y\in Y$ and $z\in Z$ where $h=xyz$.
$xyw$ is a triangle in $\partial H_1$, therefore $w\notin V(B)$ so $w\in Z$. $zw\in \partial H_1$ so $zw\in\{ab,bc\}$. Without loss of generality we may assume $zw=ab$. If $w=b$ and $z=a$, we have $x,y\in N(b)\cap V(B)=\{a',c',x',y'\}$ and $xy\in M_a$, which is impossible, because $a'y',c'x'\notin C_l$, $a'c',a'x',c'y'\notin \partial H_1$ and $x'y'\in M_b$. If $z=b$ and $w=a$, then $xy\in M_b=\{x'y'\}$ so $x',y'\in N(a)\cap V(B)=V(M_a)\cup \{c'\}$, i.e. $x'\in V(M_a)$ which contradicts Claim \ref{generalconstrclaim}.

Finally, consider the case when $\abs{h\cap V(B)}=1$. Then $h=a'bc$ or $h=c'ab$. Without loss of generality, we may assume $h=c'ab$. Then $w\in N(a)\cap N(b)\cap N(c')$. $N(b)=\{a,c,a',c',x',y'\}$, and it is easy to see that $Y\cap N(c')=\emptyset$ and $a'c'\notin \partial H_1$, therefore $y',a',c'\notin N(c')$.  By Claim \ref{generalconstrclaim}, $x'\notin V(M_a)$, so $x',a,c\notin N(a)$. Thus $N(a)\cap N(b)\cap N(c')=\emptyset$, a contradiction.


\subsection{Proof that Construction \ref{seondconstr} is sail-free}
Let $H_2$ be the hypergraph obtained by Construction \ref{seondconstr}. For each $w\in\{a,b,c\}$, let $M_w$ be the set of edges of $B_0$ with color $w$.

Let us assume for a contradiction that there exists a vertex $w\in V(H_2)$ and $h\in E(H_2)$ such that $h\subseteq N(w)$.

Then either $\abs{h\cap V(B)}=0$,  or $\abs{h\cap V(B)}=2$. 

If $\abs{h\cap V(B)}=0$, then $h=abc$. Clearly $N(a)\cap N(b)\cap N(c)=V(M_a)\cap V(M_b)\cap V(M_c)=\emptyset$, i.e. such a vertex $w$ does not exist, a contradiction.

If $\abs{h\cap V(B)}=2$  then there exist $x\in X,y\in Y$ and $z\in Z$ where $h=xyz$.
$xyw$ is a triangle in $\partial H_2$, therefore $w\notin V(B)$ so $w\in Z$. Note that $zw\in (\partial H_2)[Z]=\{ab,bc,ac\}$, so $w,z\in \{a,b,c\}$. Without loss of generality we may assume $w=b$ and $z=a$, so $h=axy$, i.e. $xy\in M_a$. Therefore $x,y\in N(w)=N(b)$ i.e., $x,y\in V(M_b)$. Then since $xy\notin M_b$ and $x,y\in V(M_b)$, both of the edges adjacent to $xy$ in $B_0$ must be in $M_b$. But then we have $3$ consecutive edges colored with only two colors, $a$ and $b$, a contradiction.


\subsection{Proof that Construction \ref{smallconstruction1} is sail-free}
Let $H_3$ be the hypergraph obtained by Construction \ref{smallconstruction1}. 
Let us assume for a contradiction that there exists $w\in V(H_3)$ and $h\in E(H_3)$ such that $h\subseteq N(w)$. Then either $\abs{h\cap V(B)}=0$,  or $\abs{h\cap V(B)}=2$. 

If $\abs{h\cap V(B)}=0$, then $h=abc$, so $w\in N(a)\cap N(b)\cap N(c)= V(M_a)\cup V(M_b)\cap V(M_c)=\emptyset$, a contradiction. 

If $\abs{h\cap V(B)}=2$  then there exist vertices $x,y\in V(B)$ and $z\in Z$ where $h=xyz$.
First let us consider the case when $w\in V(B)$. Without loss of generality we may assume $w=x_1$. $xyw$ is a triangle in $\partial H_3$, therefore $xy=x_2x_3$, i.e., $x_2x_3\in M_z$ and $z\in\{a,b,c\}$. But then since $x_2,x_3\in V(M_z)$, it is easy to see that $x_1\notin V(M_z)$. Therefore $z\notin N(x_1)=N(w)$, a contradiction.
So $w\notin V(B)$, i.e. $w\in Z$. Since $zw\in \partial H_3$, we have $zw\in \{ab,bc,ac\}$. Therefore $w,z\in \{a,b,c\}$. Without loss of generality, we may assume $w=b$ and $z=a$, so $h=axy$, i.e., $xy\in M_a$. Since $a,x,y\in N(w)=N(b)$, we have $x,y\in V(M_b)$. Then since $xy\notin M_b$ and $x,y\in V(M_b)$, both of the edges adjacent to $xy$ in $M_a\cup M_b\cup M_c$ must be in $M_b$. But since $M_a\cup M_b\cup M_c$ is a union of two properly colored vertex-disjoint triangles, this is impossible.


\subsection{Proof that Construction \ref{smallconstruction2} is sail-free}
Let $H_4$ be the hypergraph obtained by Construction \ref{smallconstruction2}. Let $B=M_a\cup M_b \cup M_c \cup M_v$, clearly $V(B)=\{x_1,x_2,x_3,y_1,y_2,y_3\}$.
Let us assume for a contradiction that there exists $w\in V(H_4)$ and $h\in E(H_4)$ such that $h\subseteq N(w)$. 
Then either $\abs{h\cap V(B)}=2$ or $\abs{h\cap V(B)}=1$. 

If $\abs{h\cap V(B)}=2$  then there exist $x,y\in V(B)$ and $z\in Z$ where $h=xyz$. Clearly, $wxy$ is a triangle in $\partial H_4$ therefore $\{w,x,y\}=\{x_1,x_2,x_3\}$ and $z\in\{a,b,c\}$. It is easy to see that $z$ is adjacent to exactly $2$ of the vertices $\{x_1,x_2,x_3\}$, so it is impossible that $N(w)$ contains $zxy$. 

Now suppose $\abs{h\cap V(B)}=1$. Then $h=y_2bc$ or $h=y_3ab$.
Without loss of generality we may assume $h=y_3ab$, so $w\in N(a)\cap N(b)\cap N(y_3)$.
$N(a)\cap N(b)=\{y_2,x_2,y_3\}$, so since $y_3y_2\notin \partial H_4$, we have $w=x_2$. Therefore $x_2y_3\in \partial H_4$, i.e., $x_2y_3\in B$ which is not the case, a contradiction.

\section*{Acknowledgement}
We are grateful to Andr\'as Gy\'arf\'as for helpful discussions (especially concerning Construction 2), for useful comments on this paper and for telling us about this problem.

\end{document}